\documentclass[reqno]{amsart}

\usepackage{amsmath}
\usepackage{amssymb}
\usepackage{hyperref}
\hypersetup{ colorlinks = true, urlcolor = blue, linkcolor = blue, citecolor = red }
\usepackage[margin=1in]{geometry}
\usepackage{xcolor}
\usepackage{comment}
\usepackage{dsfont}
\usepackage{comment}
\usepackage{cite}

\newcommand{\R}{\mathbb{R}}
\newcommand{\norm}[1]{\left\lVert#1\right\rVert}
\newcommand{\ov}{\overline}
\theoremstyle{plain}
\newtheorem{thm}{Theorem}[section]
\newtheorem{lem}{Lemma}[section]
\newtheorem{prop}{Proposition}[section]
\newtheorem{cor}{Corollary}[section]

\theoremstyle{definition}
\newtheorem{defn}{Definition}[section]

\theoremstyle{remark}

\newtheorem{rmk}{Remark}[section]

\usepackage{hyperref}
\usepackage{enumitem}
\usepackage{nameref}
\makeatletter
\let\orgdescriptionlabel\descriptionlabel
\renewcommand*{\descriptionlabel}[1]{%
  \let\orglabel\label
  \let\label\@gobble
  \phantomsection
  \edef\@currentlabel{#1\unskip}%
  \let\label\orglabel
  \orgdescriptionlabel{#1}%
}
\numberwithin{equation}{section}

\begin{document}

\title[]{$C^{1, \alpha}$ regularity for a class of singular/degenerate fully nonlinear elliptic equations with oblique boundary conditions}
\everymath{\displaystyle}

\author{Sun-Sig Byun}
\address{Department of Mathematical Sciences and Research Institute of Mathematics,
Seoul National University, Seoul 08826, Korea}
\email{byun@snu.ac.kr}

\author{Hongsoo Kim}
\address{Department of Mathematical Sciences, Seoul National University, Seoul 08826, Korea}
\email{rlaghdtn98@snu.ac.kr}

\author{SeungHyun Kim}
\address{Department of Mathematical Sciences, Seoul National University, Seoul 08826, Korea}
\email{seung27788@snu.ac.kr}

\begin{abstract}
In this paper, we establish global $C^{1, \alpha}$ regularity for viscosity solutions to a class of singular and degenerate fully nonlinear elliptic equations subject to oblique boundary conditions. Our work extends the findings in \cite{BKO25} to a broader class of equations, notably encompassing the singular case.
\end{abstract}



% 35B65 Smoothness and regularity of solutions to PDEs

% 35D40 Viscosity solutions to PDEs

% 35J25 Boundary value problems for second-order elliptic equations 

% 35J60 Nonlinear elliptic equations

% 35J70 Degenerate elliptic equations

% 35J75 Singular elliptic equations
\maketitle

\section{Introduction}
In this paper, we consider global regularity results for viscosity solutions to a class of fully nonlinear elliptic equations of the form
\begin{equation}
\label{eq} \left\{\begin{aligned}
\Phi(x, |Du|)F(D^{2}u) &=f &&\text{in } \Omega, \\
\beta\cdot Du &=g &&\text{on } \partial\Omega,
\end{aligned}\right.
\end{equation}
where $\Omega \subset \R^{n}$ is a bounded $C^{1}$ domain with boundary $\partial\Omega$, $F : \mathcal{S}^{n}\rightarrow \R$ a uniformly elliptic operator in the sense of \ref{a1} and satisfies $F(0)=0$, and $\Phi : \Omega \times (0,\infty)\rightarrow [0,\infty)$ is a continuous function that exhibits singular and degenerate behavior in the gradient, as described in \ref{a2}. Here, $f\in C(\Omega)$, $g\in C^{\alpha}(\partial\Omega)$, and $\beta\in C^{\alpha}(\partial\Omega;\R^{n})$ for some $\alpha\in(0,1)$, with $\beta$ satisfying \ref{a3}. 

We begin by discussing condition \ref{a2}. This condition was introduced in \cite{BBLL24a}, where the interior H\"older continuity of the gradient for viscosity solutions to
\begin{align}
\label{toy model}
\Phi(x, |Du|)F(D^{2}u)=f  \quad \text{ in } B_{1},  
\end{align}
was established under the assumptions that $F$ is a fully nonlinear uniformly elliptic operator and $f\in C(B_{1})\cap L^{\infty}(B_{1})$. We now present some examples satisfying \ref{a2} as follows:
\begin{enumerate}
\item  $\Phi(x,t) = t^{p} $  with $i(\Phi) = s(\Phi)=p$,
where $p>-1$.
\item  $\Phi(x,t) = t^{p(x)} $ with $i(\Phi) = \inf_{x\in\Omega} p(x)$ and 
        $s(\Phi) = \sup_{x\in\Omega} p(x)$,
where $p(\cdot) \in C(\Omega)$ and $\inf_{x\in\Omega} p(x)>-1$.
\item  $\Phi(x,t) = t^{p(x)} + a(x) t^{q(x)}$  with $i(\Phi) = \inf_{x\in\Omega} \left\{p(x),  q(x)\right\}$ and 
        $s(\Phi) = \sup_{x\in\Omega} \left\{p(x),  q(x)\right\}$,
where $p(\cdot),q(\cdot) \in C(\Omega)$, $\inf_{x\in\Omega} \left\{p(x),  q(x)\right\}>-1$ and $0\leq a(\cdot) \in C(\Omega)$.    
    \item $\Phi(x,t) = \phi(t) + a(x) \psi(t)$ where $\phi$ and $\psi$ are suitable $N$-functions and $0\leq a(\cdot) \in C(\Omega)$.
\end{enumerate}

We next turn to boundary regularity results for singular and degenerate fully nonlinear elliptic equations. In \cite{BBLL24b}, the authors established global $C^{1, \alpha}$ regularity for viscosity solutions to the Dirichlet problem associated with \eqref{toy model} on $C^{2}$ domains. For the singular case corresponding to the power-type choice $\Phi(x,t)=t^{\gamma}$ with $-1<\gamma<0$, \cite{P2008} proved global Lipschitz regularity under homogeneous Neumann boundary conditions on $C^{2}$ domains. For the degenerate case $\Phi(x,t)=t^{\gamma}$ with $\gamma>0$, \cite{BV2022} obtained global $C^{1, \alpha}$ regularity under nonhomogeneous Neumann boundary conditions on $C^{2}$ domains. For oblique boundary conditions, global $C^{1, \alpha}$ regularity was established in \cite{BKO25} for the same power-type choice $\Phi(x,t)=t^{\gamma}$ with $\gamma>0$ on $C^{1}$ domains. For further references, including those concerning uniformly elliptic equations, we refer the reader to \cite{R20, BRS26, BDF22, MS2006, LZ2018, BO26} and the references therein.

Our goal in this paper is to establish global $C^{1, \alpha}$ regularity for viscosity solutions to a general class of equations of the form \eqref{eq}, under nonhomogeneous oblique boundary conditions on $C^{1}$ domains. The novelty of our result lies in treating both the degenerate case ($i(\Phi)\geq 0$) and the singular case ($-1<i(\Phi)<0$) within a unified framework. In the degenerate regime, our result builds on \cite{BKO25} and applies to a broader class of nonlinearities beyond the power-type setting. In the singular regime, we improve upon the result of \cite{P2008}, where global Lipschitz regularity was established via the Ishii–Lion's method under homogeneous Neumann boundary conditions. We note that, in this singular regime, even the extension to nonhomogeneous Neumann boundary conditions is not straightforward. 

We now briefly explain the strategy used in the proof. To establish global $C^{1, \alpha}$ regularity, we show that a solution $u$ can be approximated by a sequence of affine functions with an error of order $r^{k(1+\alpha)}$ at each scale $r^k$ (for some fixed $r \in (0,1)$) via an iterative argument. More precisely, we obtain a sequence of affine functions $\left\{l_{k}\right\}$
satisfying 
\begin{align*}
\norm{u-l_{k}}_{L^{\infty}\left(\Omega\cap B_{r^{k}}\right)}\leq r^{k(1+\alpha)}  \qquad \text{ for all }\; k\geq 0. 
\end{align*}
The case $k=1$ follows from a compactness argument. By iterating this procedure—namely, by applying the compactness argument repeatedly to the scaled functions $v_k$ defined by 
\begin{align*}
v_{k}(x):=\frac{(u-l_{k})(r^{k}x)}{r^{k(1+\alpha)}},    
\end{align*}
we obtain the estimate for any $k$. Note that $v_{k}$ does not satisfy the original equation \eqref{eq}. Instead, it solves an equation of the form
\begin{align} \label{eq2}
\widetilde{\Phi}(|Dv_{k}-q|)F(D^{2}v_{k})=\Tilde{f},
\end{align}
where $q=-\nabla l_k/r^{k\alpha}$ and $\widetilde{\Phi}$ satisfies \ref{a2} after an appropriate scaling (see Section \ref{scaling property} for details).  Hence, to implement the compactness argument, we need an equicontinuity estimate for solutions to \eqref{eq2} that remains uniform in $q$.

In \cite{IS2016}, the authors introduced the sliding cusp method to establish H\"older estimates for uniformly elliptic equations that hold only where the gradient is large. In \cite{BKO25}, this result was extended to the oblique boundary setting. To address the difficulties arising from oblique boundary conditions, the authors combined the sliding paraboloid and cusp methods and exploited a modified scaling property adapted to the equation. As a consequence, they obtained H\"older estimates for uniformly elliptic equations that hold only where the gradient is far from some point on $C^{1}$ domains. 

By combining the above result with the structural condition \ref{a2} on $\Phi$, we obtain the desired estimate in the degenerate case. In the singular case, after a suitable transformation (see Lemma \ref{transform} for more details), a term of the form $|Du|^{\gamma}$ with $\gamma \in [0,1)$ appears, which creates a major difficulty in the analysis. More specifically, \eqref{eq2} takes the form
\begin{align*}
    \Phi\left(x, |Dv - q|\right) F(D^{2}v) = f |Dv - q|^{\gamma},
\end{align*}
for an arbitrary $q \in \R^n$. As mentioned before, we need to find a uniform analysis in $q$. However, as $|q|$ becomes very large, the right-hand side may grow indefinitely, hindering the derivation of uniform estimates independent of $q$.
To circumvent this difficulty, we restrict our attention to the regime where $|q|^\gamma\norm{f}_{L^\infty} $ is bounded; in this case, the arguments from \cite{IS2013,BKO25} remain applicable, without resorting to the Ishii-Lion's method.

We now state the main results of this paper.
\begin{thm}
\label{thm1}
Let $\Omega$ be a bounded $C^{1}$ domain and $0\in\partial\Omega$. Assume \ref{a1}--\ref{a3} (see Section~\ref{sec2}). In addition, assume that  $f\in C(\Omega)$, $g\in C^{\alpha}(\partial\Omega)$, and $\beta\in C^{\alpha}(\partial\Omega;\R^{n})$ for some $\alpha\in(0,1)$. 
Let $u$ be a viscosity solution to \eqref{eq}.
Then $u\in C^{1,\ov{\alpha}}(\ov{\Omega\cap B_{1/2}})$ with the estimate
\begin{align*}
\norm{u}_{C^{1,\ov{\alpha}}\left(\ov{\Omega\cap B_{1/2}}\right)}\leq c\left(\norm{u}_{L^{\infty}(\Omega\cap B_{1})}+\norm{f}_{L^{\infty}(\Omega\cap B_{1})}^{1/(1+i(\Phi))}+\norm{g}_{C^{\alpha}(\partial\Omega\cap B_{1})}\right),    
\end{align*}
where 
\begin{equation}
\label{alpha range}
\ov{\alpha}\in 
\begin{cases}
(0,\alpha_{0})\cap\left(0, \frac{1}{1+s(\Phi)}\right]\cap(0, \alpha]  \quad &\text{if } i(\Phi)\geq0, \\
(0,\alpha_{0})\cap\left(0, \frac{1}{1+s(\Phi)-i(\Phi)}\right]\cap(0,\alpha] \quad &\text{if } -1<i(\Phi)<0,
\end{cases}    
\end{equation}
$c=c(n, \lambda, \Lambda, i(\Phi), M, \nu_{0}, \delta_{0},[\beta]_{C^{\alpha}(\partial\Omega\cap B_{1})}, C^{1} \text{ modulus of }\partial\Omega\cap B_{1}, \ov{\alpha})$, and $\alpha_{0}$ will be specified in  Remark~\ref{optimal exponent}.
\end{thm}

\begin{rmk}
\label{optimal exponent}
The constant $\alpha_0 = \alpha_0(n,\lambda,\Lambda,\delta_0) \in (0,1]$ in the statement of the theorem denotes the optimal exponent of regularity theory for a homogeneous equation
with a constant oblique boundary condition. More precisely, any viscosity
solution $h$ of
\begin{equation*}
\left\{\begin{aligned}
F(D^{2}h) &=0 &&\text{in } B_{1}^{+}, \\
\beta_0\cdot Dh&=0 &&\text{on } T_{1},
\end{aligned}\right.
\end{equation*}
where $\beta_0$ is a constant vector satisfying
\ref{a3}, belongs to the class $C^{1,\alpha_0}_{\text{loc}}(B^+_1)$ and satisfies the estimate
\begin{align*}
\norm{h}_{C^{1,\alpha_0}(B^+_{3/4})} \leq
C_e\norm{h}_{L^\infty(B^+_1)},
\end{align*}
for some constant $C_e = C_e(n,\lambda,\Lambda,\delta_0)>1$. The authors in \cite{LZ2018} proved that there exists a universal constant $\alpha_0\in (0,1)$ for any uniformly elliptic $F$ and that $\alpha_0=1$ for any convex $F$.   
\end{rmk}

As a direct consequence of Theorem \ref{thm1}, we have the following result.
\begin{cor}
Let $\Omega$ be a bounded $C^{1}$ domain and $0\in\partial\Omega$. Assume \ref{a1}--\ref{a3}. In addition, assume that  $f\in C(\Omega)$, $g, h\in C^{\alpha}(\partial\Omega)$, and $\beta\in C^{\alpha}(\partial\Omega;\R^{n})$ for some $\alpha\in(0,1)$. 
Let $u$ be a viscosity solution to
\begin{equation*}
\left\{\begin{aligned}
\Phi(x, |Du|)F(D^{2}u) &=f &&\text{in } \Omega, \\
\beta\cdot Du+hu&=g &&\text{on } \partial\Omega.
\end{aligned}\right.
\end{equation*}
Then $u\in C^{1,\ov{\alpha}}(\ov{\Omega\cap B_{1/2}})$ with the estimate
\begin{align*}
\norm{u}_{C^{1,\ov{\alpha}}\left(\ov{\Omega\cap B_{1/2}}\right)}\leq c\left(\norm{u}_{L^{\infty}(\Omega\cap B_{1})}+\norm{f}_{L^{\infty}(\Omega\cap B_{1})}^{1/(1+i(\Phi))}+\norm{g}_{C^{\alpha}(\partial\Omega\cap B_{1})}\right),    
\end{align*}
where 
\begin{equation*}
\ov{\alpha}\in 
\begin{cases}
(0,\alpha_{0})\cap\left(0, \frac{1}{1+s(\Phi)}\right]\cap(0, \alpha]  \quad &\text{if } i(\Phi)\geq0, \\
(0,\alpha_{0})\cap\left(0, \frac{1}{1+s(\Phi)-i(\Phi)}\right]\cap(0,\alpha] \quad &\text{if } -1<i(\Phi)<0,
\end{cases}    
\end{equation*}
and
$c=c(n, \lambda, \Lambda, i(\Phi), M, \nu_{0}, \delta_{0},[\beta]_{C^{\alpha}(\partial\Omega\cap B_{1})}, \norm{h}_{C^{\alpha}(\partial\Omega\cap B_{1})}, C^{1} \text{ modulus of }\partial\Omega\cap B_{1}, \ov{\alpha})$.
\end{cor}

This paper is organized as follows. Section \ref{sec2} introduces the basic notation and collects preliminary results used throughout the paper. Section \ref{sec3} is devoted to prove improvement of flatness results using the compactness method. The proof of Theorem \ref{thm1} is presented in Section \ref{sec4}.

\section{Preliminaries}
\label{sec2}
\subsection{Notation, definitions and main assumptions} For $r>0$, we denote by $B_r = \{x \in \R^n : |x| < r\}$ the open ball in $\R^n$ centered at the origin with radius $r$. We write $x=(x',x_n)\in \R^n$, where $x'\in \R^{n-1}$. We also set
\[
B_r^+ := B_r \cap \{x_n > 0\}, \quad T_r := B_r \cap \{x_n = 0\},
\]
where $B_r^+$ is the upper half-ball and $T_r$ denotes its flat boundary. For $x_0 \in \R^n$, we denote by $B_r(x_0) := B_r + x_0$ the ball centered at $x_0$. For a domain $\Omega \subset \R^n$ and $x_0 \in \R^n$, we set
\[
\Omega_r := \Omega \cap B_r, \quad \partial \Omega_r := \partial \Omega \cap B_r,
\]
and
\[
\Omega_r(x_0) := \Omega \cap B_r(x_0), \quad \partial \Omega_r(x_0) := \partial \Omega \cap B_r(x_0).
\]
Let $\mathcal{S}^{n}$ denote the space of real $n\times n$ symmetric matrices. For parameters $0<\lambda\leq\Lambda$, the Pucci extremal operators $\mathcal{M}^{\pm}_{\lambda, \Lambda}: \mathcal{S}^{n}\rightarrow \R$ are defined as 
\begin{align*}
\mathcal{M}^{+}_{\lambda, \Lambda}(M)=  \Lambda\sum_{e_i>0}e_{i}+\lambda\sum_{e_{i}<0}e_{i} \quad \text{and} \quad \mathcal{M}^{-}_{\lambda, \Lambda}(M)=  \lambda\sum_{e_{i}>0}e_{i}+\Lambda\sum_{e_{i}<0}e_{i}, 
\end{align*}
where $e_{i}=e_{i}(M)$ are the eigenvalues of $M\in\mathcal{S}^{n}$. 
For a measurable function $h : \Omega \to \R$ and $\gamma \in (0,1]$, we define
\[
[h]_{C^{\gamma}(\Omega)} := \sup_{\substack{x,y \in \Omega \\ x \neq y}} \frac{|h(x) - h(y)|}{|x - y|^{\gamma}}.
\]
We denote by $USC(\overline{\Omega})$ the set of upper semicontinuous functions on $\overline{\Omega}$ and by $LSC(\overline{\Omega})$ the set of lower semicontinuous functions on $\overline{\Omega}$. 

We assume that $\Omega$ is a bounded $C^{1}$ domain. 
Without loss of generality, we may assume that $0 \in \partial \Omega_1$. 
Then $\partial \Omega_1$ can be represented as the graph of a function 
$\varphi = \varphi_{\Omega} : T_1 \to \mathbb{R}$ with $\varphi \in C^1$. 
That is,
\begin{equation}
\label{boundary repre}
\begin{aligned}
\partial\Omega_1 &= \{ (x',x_n) \in B_1 : x_n = \varphi(x') \}, \\
\Omega_1 &= \{ (x',x_n) \in B_1 : x_n > \varphi(x') \}.
\end{aligned}
\end{equation}
We write $\beta(x') := \beta(x', \varphi(x'))$, viewing $\beta$ as a function on $T_1$.

We introduce the notion of an almost non-decreasing function.
\begin{defn}
\label{non decrea}
We say that a function $f:(0, \infty)\to\R$ is almost non-decreasing (resp., almost non-increasing) with constant $M>0$ in $(0, \infty)$ if it satisfies
\begin{align*}
f(t)\leq Mf(s) \quad (\text{resp.,} \;\; f(t)\geq Mf(s)) \quad \text{for all } 0<t\leq s.  
\end{align*}
\end{defn}

We now state the main assumptions of the paper.
\begin{description}
    \item[(A1)\label{a1}] $F$ is $(\lambda, \Lambda)$-uniformly elliptic, i.e., 
    \begin{align*}
    \mathcal{M}^{-}_{\lambda, \Lambda}(X-Y)\leq F(X)-F(Y)\leq\mathcal{M}^{+}_{\lambda, \Lambda}(X-Y),
    \end{align*}
    for every $X, Y\in \mathcal{S}^{n}$, and $F(0)=0$.
    
    \item[(A2)\label{a2}] $\Phi : \Omega\times (0,\infty)\rightarrow [0,\infty)$ is a continuous map satisfying the following properties: 
    \begin{enumerate}
        \item[1.] There exist constants $s(\Phi)\geq i(\Phi)> -1$ such that the map $\displaystyle t\mapsto \frac{\Phi(x,t)}{t^{i(\Phi)}}$ is almost non-decreasing (see Definition \ref{non decrea}) with constant $M\geq 1$ in $(0,\infty)$, and the map $\displaystyle t\mapsto \frac{\Phi(x,t)}{t^{s(\Phi)}}$ is almost non-increasing with constant $M\geq 1$ in $(0,\infty)$ for all $x\in \Omega$.
        \item[2.] There exist constants $0<\nu_0\leq \nu_1$ such that $\displaystyle \nu_{0} \leq \Phi(x,1) \leq \nu_{1}$ for all $x\in \Omega$.
    \end{enumerate}
    
    \item[(A3)\label{a3}] There exists a positive constant $\delta_{0}>0$ such that 
    \begin{align*}
    \beta\cdot \mathbf{n}\geq\delta_{0} \quad \text{and} \quad \norm{\beta}_{L^{\infty}(\partial\Omega)}\leq 1,  
    \end{align*}
    where $\mathbf{n}$ is the inner normal vector of $\partial\Omega$.     
\end{description}

\indent We give the definition of viscosity solutions to the equation
\begin{align*}
G(D^{2}u,Du,x)=f \quad \text{in} \quad \Omega,
\end{align*}
with oblique boundary conditions, where $G:\mathcal{S}^{n}\times\mathbb{R}^{n}\times\Omega \to \mathbb{R}$ and $f \in C(\Omega)\cap L^{\infty}(\Omega)$. 
For the definition of viscosity solutions with oblique boundary conditions, we refer to \cite{BKO25} and the references therein, while for the definition in the singular setting we refer to \cite{BBLL24a} and the references therein.

\begin{defn} 
	A function $u\in USC(\overline{\Omega})$ (resp., $u\in LSC(\overline{\Omega})$) is called a viscosity subsolution (resp., supersolution) of  
\begin{equation}
\left\{\begin{aligned}
\label{def}
    G(D^{2}u,Du,x)&=f &&\text{in } \Omega,\\
    \beta\cdot Du&=g &&\text{on } \partial\Omega,\\
\end{aligned}\right.
\end{equation}
if the following hold:
\begin{enumerate}
\item[(i)] If $x_0\in \Omega$, then either

\begin{enumerate}
\item[(a)] There exists $\delta>0$ such that $B_{\delta}(x_{0})\subset \Omega$, 
$u$ is constant in $B_{\delta}(x_{0})$, and
\[
f(x) \leq  0 \quad (\text{resp.,} \;\; f(x)\geq 0) \quad\quad \text{for all } x\in B_{\delta}(x_{0}).
\]

\item[(b)] For all $\varphi\in C^{2}(\overline{\Omega})$ such that 
$u-\varphi$ has a local maximum (resp., minimum) at $x_0$ and 
$D\varphi(x_{0})\neq 0$, then
\[
G(D^{2}\varphi(x_0),D\varphi(x_0),x_{0})
\geq  f(x_{0}) \quad (\text{resp.,} \;\; G(D^{2}\varphi(x_0),D\varphi(x_0),x_{0})
\leq  f(x_{0})).
\]
\end{enumerate}
\item[(ii)] If $x_0\in \partial\Omega$, for all $\varphi\in C^{2}(\overline{\Omega})$ such that $u-\varphi$ has a local maximum (resp., minimum) at $x_0$, then 
	\begin{align*}
		\beta(x_{0})\cdot D\varphi(x_0)\leq  \;g(x_{0}) \quad (\text{resp.,} \;\; \beta(x_{0})\cdot D\varphi(x_0)\geq  \;g(x_{0})).
	\end{align*}	
\end{enumerate}
We say that $u\in C(\overline{\Omega})$ is a viscosity solution of \eqref{def} if $u$ is is both a viscosity subsolution and supersolution.
\end{defn}

\subsection{Known results}
We recall a result on the boundary H\"older regularity for uniformly elliptic equations on $C^1$ domains that hold only where the gradient is far from some point.

We fix $0<\lambda \leq \Lambda$, $M \in \mathcal{S}^{n}$, and $q \in \mathbb{R}^{n}$, and define
\begin{align*}
\mathcal{P}^{+}_{\lambda, \Lambda}(M,q):&=\mathcal{M}^{+}_{\lambda, \Lambda}(M)+\Lambda|q|, \\
\mathcal{P}^{-}_{\lambda, \Lambda}(M,q):&=\mathcal{M}^{-}_{\lambda, \Lambda}(M)-\Lambda|q|.
\end{align*}

\begin{lem}\cite[Theorem 3.1]{BKO25}
\label{Holder original}
Let $\Omega$ be a bounded $C^{1}$ domain whose boundary is represented as in \eqref{boundary repre}. Assume that \ref{a3} holds. In addition, assume that $g\in C(\partial\Omega)$ and $\beta\in C(\partial\Omega;\R^{n})$. Let $u\in C(\ov{\Omega_{1}})$ be a viscosity solution to 
\begin{equation*}
\left\{\begin{aligned}
\mathcal{P}^{+}_{\lambda, \Lambda}(D^{2}u,Du)&\geq -C_{0} && \text{in } \left\{|Du-q|>\theta\right\}\cap\Omega_{1}, \\ 
\mathcal{P}^{-}_{\lambda, \Lambda}(D^{2}u,Du) &\leq C_{0} && \text{in } \left\{|Du-q|>\theta\right\}\cap\Omega_{1}, \\
\beta\cdot Du &=g &&\text{on } \partial\Omega_{1},\\
\norm{u}_{L^{\infty}(\Omega_{1})}&\leq 1, \\
\norm{g}_{L^{\infty}(T_{1})}&\leq C_{0},
\end{aligned}\right.
\end{equation*}    
where $q\in\R^{n}$ and $0<\theta\leq1$.
Then there exists $\mu=\mu(\delta_{0}) \in(0, \delta_{0}/2]$ such that, if $[\varphi]_{C^{1}}\leq \mu$, then $u\in C^{\alpha}(\ov{\Omega_{1/2}})$ for some $\alpha=\alpha(n, \lambda, \Lambda,\delta_{0})\in(0,1)$ with the estimate
\begin{align*}
\norm{u}_{C^{\alpha}(\ov{\Omega_{1/2}})}\leq C,
\end{align*}
where $C=C(n, \lambda, \Lambda, \delta_{0}, C_{0})$ does not depend on $q$.
\end{lem}

We next introduce a transformation that converts the equation from the singular regime to the degenerate regime. 

\begin{lem} \cite[Proposition 2.1]{BBLL24a}
\label{transform}
 Let $u$ be a viscosity solution to 
\begin{align*}
\Phi(x, |Du|)F(D^{2}u)=f \quad \text{in } \Omega,    
\end{align*}
where $F$ satisfies \ref{a1}, $\Phi$ satisfies \ref{a2} with $-1<i(\Phi)<0$, and $f\in C(\Omega)\cap L^{\infty}(\Omega)$. 
If we define 
\begin{align*}
\widetilde{\Phi}(x,t):=\frac{\Phi(x,t)}{t^{i(\Phi)}} \quad \text{in } \Omega\times(0, \infty),    
\end{align*} 
then $u$ solves 
\begin{align*}
\widetilde{\Phi}(x,|Du|)F(D^{2}u)=f|Du|^{-i(\Phi)} 
\quad \text{in } \Omega, 
\end{align*}
in the viscosity sense. Moreover, $\widetilde{\Phi}$ satisfies \ref{a2} in $\Omega\times(0, \infty)$ with the parameters 
\begin{align*}
    i(\widetilde{\Phi})=0, \quad s(\widetilde{\Phi})=s(\Phi)-i(\Phi)\geq0, \quad \widetilde{M}=M, \quad \widetilde{\nu}_{0}={\nu}_{0}, \quad\text{ and } \quad \widetilde{\nu}_{1}={\nu}_{1}.
\end{align*}
\end{lem}

\begin{rmk}
\label{unified}
In light of Lemma \ref{transform}, we can treat the singular and degenerate regimes within a unified framework. Specifically, instead of the original problem \eqref{eq}, it suffices to consider the modified problem
\begin{equation}
\label{problem1}
\left\{\begin{aligned}
\Phi(x,|Du|)F(D^{2}u)&=f|Du|^{\gamma} &&\text{in } \Omega_{1},\\
\beta\cdot Du&=g &&\text{on } \partial\Omega_{1},
\end{aligned}\right.
\end{equation}
under the assumptions \ref{a1}--\ref{a3}, but with the restrictions $i(\Phi)\geq 0$ and $\gamma \in [0,1)$. (By a slight abuse of notation, we continue to use $\Phi$ and $f$ for the transformed quantities.) 

To clarify the correspondence between \eqref{eq} and \eqref{problem1}, let $i_{orig}$ and $s_{orig}$ denote the parameters of $\Phi$ in the original equation \eqref{eq}. Then the parameters $(i(\Phi), s(\Phi), \gamma)$ in \eqref{problem1} are determined as follows:
\begin{itemize}
    \item In the \textbf{degenerate case} ($i_{orig} \geq 0$), we simply have 
    $$ (i(\Phi), s(\Phi), \gamma) = (i_{orig}, s_{orig}, 0). $$
    \item In the \textbf{singular case} ($-1 < i_{orig} < 0$), Lemma \ref{transform} yields 
    $$ (i(\Phi), s(\Phi), \gamma) = (0, s_{orig}-i_{orig}, -i_{orig}). $$
\end{itemize}
\end{rmk}

\subsection{Scaling property}
\label{scaling property}
In this subsection, we investigate the scaling property of viscosity solutions to 
\begin{align}
\label{scale eq}
\Phi(x, |Du|)F(D^{2}u)=f|Du|^{\gamma} \quad \text{in } \Omega,    
\end{align}
where $\Phi$ satisfies \ref{a2}, $\gamma \in [0, 1)$ and $f\in C(\Omega)\cap L^{\infty}(\Omega)$.
Let $u$ be a viscosity solution to \eqref{scale eq}. For $B_{r}(x_{0})\subset\Omega$ and $\kappa>0$, define $$\widetilde{u}(x):=\frac{u(rx+x_{0})}{\kappa} \quad \text{in } B_{1}.$$ 
Then $\widetilde{u}$ satisfies
\begin{align*}
\widetilde{\Phi}(x,|D\widetilde{u}|)\widetilde{F}(D^{2}\widetilde{u})=\widetilde{f} \quad \text{in } B_{1},    
\end{align*}
in the viscosity sense, where
\begin{align*}
\widetilde{\Phi}(x,t) &:= \frac{\Phi\left(rx+x_{0},\frac{\kappa}{r}t\right)}{\Phi\left(rx+x_{0}, \frac{\kappa}{r}\right)} \quad \text{in } B_{1}\times(0, \infty), \\
\widetilde{F}(M) &:= \frac{r^{2}}{\kappa}F\left(\frac{\kappa}{r^{2}}M\right) \quad \text{in } \mathcal{S}^{n}, \\
\widetilde{f}(x) &:= \frac{r^{2-\gamma}f(rx+x_{0})}{\kappa^{1-\gamma}\Phi\left(rx+x_{0}, \frac{\kappa}{r}\right)} \quad \text{in } B_{1}.
\end{align*}
Note that by direct calculation, $\widetilde{\Phi}$ satisfies \ref{a2} in $B_{1} \times (0, \infty)$ with the parameters 
\begin{align*}
    i(\widetilde{\Phi})=i(\Phi), \quad s(\widetilde{\Phi})=s(\Phi), \quad \widetilde{M}=M, \quad \text{and} \quad \widetilde{\nu}_{0}=\widetilde{\nu}_{1}=1.
\end{align*}
In addition, $\widetilde{F}$ satisfies \ref{a1} with the same ellipticity constants $\lambda$ and $\Lambda$ as $F$.

\section{Improvement of flatness}
\label{sec3}

\subsection{Boundary H\"older regularity}

As discussed in the preceding Remark \ref{unified}, we henceforth focus on the unified problem \eqref{problem1} under the assumptions \ref{a1}--\ref{a3}, where the gradient singularity has been removed (i.e., $i(\Phi) \geq 0$ and $\gamma \in [0,1)$).

We first establish the boundary H\"older regularity of solutions to \eqref{problem1} by employing Lemma \ref{Holder original}. This regularity result plays a crucial role in the compactness argument, which will subsequently be used to prove the improvement of flatness.  

\begin{lem}
\label{holder}
Let $\Omega$ be a bounded $C^{1}$ domain whose boundary is represented as in \eqref{boundary repre}. Assume \ref{a1}--\ref{a3}. Let $i(\Phi)\geq0$, $0\leq\gamma<1$ and $A>0$. Assume that  $f\in C(\Omega)\cap L^{\infty}(\Omega)$, $g\in C(\partial\Omega)$ and $\beta\in C(\partial\Omega;\R^{n})$. Let $u$ satisfy $|u|\leq1$ in $\Omega_{1}$ and be a viscosity solution to 
\begin{equation*}
\left\{\begin{aligned}
\Phi\left(x, \left|Du-q\right|\right)F(D^{2}u) &=f|Du-q|^{\gamma} &&\text{in } \Omega_{1}, \\
\beta\cdot Du&=g &&\text{on } \partial\Omega_{1},
\end{aligned}\right.
\end{equation*}
where $(1+|q|^{(\gamma-i(\Phi))_+})\norm{f}_{L^\infty}\leq A$. 
Then there exists $\mu=\mu(\delta_{0}) \in(0, \delta_{0}/2]$ such that, if $[\varphi]_{C^{1}}\leq \mu$, then we have $u\in C^{\alpha}(\ov{\Omega_{1/2}})$ for some $\alpha=\alpha(n, \lambda, \Lambda,\delta_{0})\in(0,1)$ with the estimate
\begin{align*}
\norm{u}_{C^{\alpha}(\ov{\Omega_{1/2}})}\leq C,
\end{align*}
where $C=C(n, \lambda, \Lambda, i(\Phi), M, \nu_{0}, \delta_{0} ,\gamma, A)$ does not depend on $q$.
\end{lem}

\begin{proof}
If $\left|Du-q\right|>1$, using \ref{a1} and \ref{a2}, we have
\begin{align*}
\mathcal{P}_{\lambda, \Lambda}^{+}(D^{2}u, Du)=\mathcal{M}^{+}_{\lambda, \Lambda}(D^{2}u)+\Lambda|Du|\geq F(D^{2}u)+\Lambda|Du|\geq \frac{f|Du-q|^{\gamma}}{\Phi(x,|Du-q|)}+\Lambda|Du|     
\end{align*}
and
\begin{align*}
\frac{\Phi(x, |Du-q|)}{|Du-q|^{i(\Phi)}}\geq \frac{1}{M}\Phi(x,1)\geq \frac{\nu_{0}}{M}.  
\end{align*}
Thus, we get
\begin{align*}
\mathcal{P}_{\lambda, \Lambda}^{+}(D^{2}u, Du)\geq -\frac{M}{\nu_{0}}\norm{f}_{L^{\infty}}\left|Du-q\right|^{\gamma-i(\Phi)}+\Lambda|Du| \quad \text{in } \left\{\left|Du-q\right|>1\right\}\cap\Omega_{1}.    
\end{align*}
If $i(\Phi)\geq\gamma$, 
\begin{align*}
\mathcal{P}_{\lambda, \Lambda}^{+}(D^{2}u, Du)\geq -\frac{M}{\nu_{0}}\norm{f}_{L^{\infty}}\geq -\frac{M}{\nu_{0}}A\quad \text{in } \left\{\left|Du-q\right|>1\right\}\cap\Omega_{1}.    
\end{align*}
If $i(\Phi)<\gamma$, we obtain
\begin{align*}
\mathcal{P}_{\lambda, \Lambda}^{+}(D^{2}u, Du)&\geq\left(\Lambda|Du|-\frac{M}{\nu_{0}}\norm{f}_{L^{\infty}}|Du|^{\gamma-i(\Phi)}\right)-\frac{M}{\nu_{0}}\norm{f}_{L^{\infty}}|q|^{\gamma-i(\Phi)} \\
&\geq \left(\Lambda|Du|-\frac{M}{\nu_{0}}A|Du|^{\gamma-i(\Phi)}\right)-\frac{M}{\nu_{0}}A\\
&\geq -C(\Lambda, \nu_{0}, M, \gamma, i(\Phi), A) \qquad \text{in } \left\{\left|Du-q\right|>1\right\}\cap\Omega_{1},
\end{align*}
where we have used $0<\gamma-i(\Phi)\leq \gamma<1$. Combining both cases, we deduce that
\begin{align*}
\mathcal{P}_{\lambda, \Lambda}^{+}(D^{2}u, Du) \geq -C(\Lambda, \nu_{0}, M, \gamma, i(\Phi), A) \quad \text{in } \{|Du-q|>1\}\cap\Omega_{1}.
\end{align*}
By an analogous argument, we can also show that
\begin{align*}
\mathcal{P}_{\lambda, \Lambda}^{-}(D^{2}u, Du) \leq C(\Lambda, \nu_{0}, M, \gamma, i(\Phi), A) \quad \text{in } \{|Du-q|>1\}\cap\Omega_{1}.
\end{align*}
Therefore, the desired conclusion follows from Lemma \ref{Holder original}.
\end{proof}

\subsection{Improvement of flatness}

The following lemma provides an approximation result via a compactness argument, which serves as a key ingredient for the improvement of flatness.
\begin{lem}
\label{approx}
Let $\Omega$ be a bounded $C^{1}$ domain whose boundary is represented as in \eqref{boundary repre}. Assume \ref{a1}--\ref{a3}. Let $i(\Phi)\geq0$, $0\leq\gamma<1$ and $A>0$. Assume that  $f\in C(\Omega)\cap L^{\infty}(\Omega)$, $g\in C^{\alpha}(\partial\Omega)$, and $\beta\in C^{\alpha}(\partial\Omega;\R^{n})$ for some $\alpha\in(0,1)$. Let $u$ satisfy $|u|\leq1$ in $\Omega_{1}$ and be a viscosity solution to 
\begin{equation*}
\left\{\begin{aligned}
\Phi\left(x, \left|Du-q\right|\right)F(D^{2}u) &=f|Du-q|^{\gamma} &&\text{in } \Omega_{1}, \\
\beta\cdot Du&=g &&\text{on } \partial\Omega_{1}.
\end{aligned}\right.
\end{equation*}
Given $\delta>0$, there exists $\epsilon=\epsilon(n, \lambda, \Lambda, i(\Phi), M, \nu_{0}, \delta_{0}, \gamma, \delta)>0$ such that if
\begin{align*}
(1+|q|^{(\gamma-i(\Phi))_+})\norm{f}_{L^{\infty}}, \; \norm{g}_{L^{\infty}}, \; [\beta]_{C^{\alpha}}, \;[\varphi]_{C^{1}}\leq\epsilon,    
\end{align*}
then there exists a function $h\in C^{1, \alpha_{0}}\left(\ov{B_{3/4}^{+}}\right)$ such that
\begin{align*}
\norm{u-h}_{L^{\infty}\left(\Omega_{1/2}\right)}\leq \delta, \quad \beta(0) \cdot Dh(0)=0, \quad \text{and} \quad \norm{h}_{C^{1, \alpha_{0}}\left(\ov{B_{3/4}^{+}}\right)}\leq \ov{C_e},  \end{align*}
where $\ov{C_e}=3C_e$, and $C_e$ is the universal constant given in Remark \ref{optimal exponent}.
\end{lem}

\begin{proof}
For a contradiction, assume that the conclusion does not hold. Then there exist a $\ov{\delta}>0$ and sequences $\{\Phi_{k}\}^{\infty}_{k=1}, \{F_{k}\}^{\infty}_{k=1}, 
 \{f_{k}\}^{\infty}_{k=1}, \{g_{k}\}^{\infty}_{k=1}, \{\beta_{k}\}^{\infty}_{k=1}, \{q_{k}\}^{\infty}_{k=1}, \{u_{k}\}^{\infty}_{k=1}$, and $\{\Omega_{k}\}^{\infty}_{k=1}$ such that 
\begin{description}
\item[(C1)] $\Phi_{k} \in C((\Omega_{k})_{1}\times (0,\infty))$ satisfies \ref{a2} with the parameters $i(\Phi)$, $s(\Phi_{k})$, $M$, $\nu_{0}$, and $\nu_{1,k}$;
\item[(C2)\label{c2}] $F_{k}$ is uniformly $(\lambda,\Lambda)$-elliptic and $F_{k}(0)=0$;
\item[(C3)\label{c3}] $f_{k}\in C\left(\Omega_{k}\right) \cap L^{\infty}(\Omega_{k})$ with $(1+|q_k|^{(\gamma-i(\Phi))_+})\norm{f_{k}}_{L^{\infty}((\Omega_{k})_{1})}\leq 1/k$, $g_{k}\in C^{\alpha}(\partial\Omega_{k})$ 
 with $\norm{g_{k}}_{L^{\infty}(\partial(\Omega_{k})_{1})}\leq 1/k$, $\beta_{k}\in C^{\alpha}\left(\partial\Omega_{k};\R^{n}\right)$  with $[\beta_{k}]_{C^{\alpha}(\partial(\Omega_{k})_{1})}\leq 1/k$, $\varphi_{k}\in C^{1}(T_{1})$ with $[\varphi_{k}]_{C^{1}(T_{1})}\leq 1/k$;
\item[(C4)] $u_{k}$ with $\norm{u_{k}}_{L^{\infty}((\Omega_{k})_{1})}\leq 1$ solves
\begin{equation}
\left\{\begin{aligned}
\label{2}
\Phi_{k}\left(x, \left|Du_{k}-q_k\right|\right)F_{k}(D^{2}u_{k})&=f_{k}|Du_{k}-q_{k}|^{\gamma} &&\text{in } (\Omega_{k})_{1},\\
\beta_{k}\cdot Du_{k}&=g_k &&\text{on } \partial(\Omega_{k})_{1},\\
\end{aligned}\right.
\end{equation}
but $\norm{u_{k}-h}_{L^{\infty}}>\ov{\delta}$ for any $k\in \mathbb{N}$ and $h\in C^{1,\alpha_{0}}\left(\overline{B_{3/4}^{+}}\right)$ satisfying $\beta_k(0) \cdot Dh(0)=0$ and $\norm{h}_{C^{1, \alpha_{0}}\left(\overline{B_{3/4}^{+}}\right)}\leq \ov{C_e}$.
\end{description}
Note that by \ref{c2} and the Arzel\`a--Ascoli theorem, $F_{k}$ converges locally uniformly to some $F_{\infty}$ satisfying \ref{a1}.
By the boundary H\"older estimates in Lemma \ref{holder} and the Arzel\`a--Ascoli theorem, there exists a function $u_{\infty}\in C(\ov{B_{1}^{+}})$ such that $u_{k}$ converges to $u_{\infty}$ locally uniformly up to a subsequence.
Additionally, $\beta_{k}\to \beta_\infty$ for some constant vector $\beta_\infty \in \R^n$ and $(\Omega_{k})_{1}\to B_{1}^{+}$.
Now we claim that $u_{\infty}$ is a viscosity solution to
\begin{equation*}
\left\{\begin{aligned}
F_{\infty}(D^{2}u_{\infty}) &=0 &&\text{in } B_{1}^{+}, \\
\beta_\infty\cdot Du_{\infty}&=0 &&\text{on } T_{1}.
\end{aligned}\right.
\end{equation*}
Let 
\begin{align*}
\psi(x)=\frac{1}{2}M(x-y)\cdot(x-y)+b\cdot(x-y)+u_{\infty}(y)    
\end{align*}
be a quadratic polynomial touching $u_{\infty}$ from below at a point $y$. We may assume that $\psi$ touches $u_{\infty}$ strictly from below at $y$, and for simplicity, that $|y|=u_{\infty}(y)=0$.  Then there exists a sequence $x_{k}\to 0$ as $k\to\infty$ such that $u_{k}-\psi$ has a local minimum at $x_{k}$. Note that $D\psi(x_{k})\to b$ and $D^{2}\psi(x_{k})=M$. Since $u_{k}$ satisfies \eqref{2} in the viscosity sense, we have
\begin{align}
\label{vis ineq}
\Phi_{k}\left(x_{k}, \left|D\psi(x_{k})-q_{k}\right|\right)F_{k}(D^{2}\psi(x_{k}))\leq f_{k}(x_{k})|D\psi(x_{k})-q_{k}|^{\gamma}.    
\end{align}
We divide the argument depending on the boundedness of the sequence $\{q_{k}\}^{\infty}_{k=1}$.

\medskip
\noindent Case I: The sequence $\{q_{k}\}^{\infty}_{k=1}$ is unbounded. In this case, we may assume $|q_{k}|\to \infty$ as $k\to\infty$ (up to a subsequence). Then we have
\begin{align*}
\frac{3}{2}|q_{k}|\geq\left|D\psi(x_{k})-q_{k}\right|\geq\frac{1}{2} |q_{k}|\geq1,    \end{align*}
for sufficiently large $k$. From \eqref{vis ineq}, \ref{a2} and \ref{c3}, we obtain
\begin{align*}
 F_{k}(D^{2}\psi(x_{k})) &\leq\frac{f_{k}(x_{k})\left|D\psi(x_{k})-q_{k}\right|^{\gamma}}{\Phi_{k}\left(x_{k}, \left|D\psi(x_{k})-q_{k}\right|\right)} \\&\leq \frac{M}{\nu_{0}}\norm{f_{k}}_{L^{\infty}}\left|D\psi(x_{k})-q_{k}\right|^{\gamma-i(\Phi)} \\  
&\leq \frac{M}{\nu_{0}}\norm{f_{k}}_{L^{\infty}}\left|\frac{3}{2}q_{k}\right|^{(\gamma-i(\Phi))_{+}} \\
&\leq \frac{M}{\nu_{0}}\left(\frac{
3}{2}\right)^{(\gamma-i(\Phi))_{+}}\frac{1}{k}.
\end{align*}
Thus
\begin{align*}
F_{\infty}(M)=\lim_{k\to\infty}F_{k}(D^{2}\psi(x_{k}))\leq0.    
\end{align*}
\noindent Case II: The sequence $\{q_{k}\}^{\infty}_{k=1}$ is bounded. 
In this case, we may assume $|q_{k}|\to \xi$ as $k\to \infty$ (up to a subsequence). Then we have
\begin{align*}
\left|D\psi(x_{k})-q_{k}\right|\leq 2(|b|+|\xi|),    \end{align*}
for sufficiently large $k$. From \ref{c3}, we get
\begin{align*}
f_{k}(x_{k})|D\psi(x_{k})-q_{k}|^{\gamma}\leq \norm{f_{k}}_{L^{\infty}} 2^{\gamma}(|b|+|\xi|)^{\gamma}\leq2^{\gamma}(|b|+|\xi|)^{\gamma}/k \to 0.   
\end{align*}
Then, from \eqref{vis ineq}, we can follow the arguments in \cite[Lemma 4.1]{BBLL24a} with minor modifications to deduce that $F_{\infty}(M)\leq0$. 

In both cases, we have $F_{\infty}(M)\leq0$, implying that $u_{\infty}$ is a supersolution. The subsolution case can also be proved in a similar manner. Finally, applying the stability property of viscosity solutions to the oblique boundary condition yields the claim.

Therefore, since $u_\infty$ is a solution to the homogeneous equation with a constant oblique boundary condition, Remark \ref{optimal exponent} yields
\begin{equation*}
    \norm{u_\infty}_{C^{1, \alpha_{0}}\left(\overline{B_{3/4}^{+}}\right)} \leq C_{e}\norm{u_{\infty}}_{L^{\infty}(B_{1}^{+})}\leq C_e.
\end{equation*}
Moreover, noting that $|Du_\infty(0)| \leq C_e$, we define
\begin{equation*}
    a_{k} := \frac{\beta_{k}(0)\cdot Du_{\infty}(0)}{|\beta_{k}(0)|^{2}}\beta_{k}(0).
\end{equation*}
It follows that $|a_k| \leq C_e$ and
\begin{equation*}
    \beta_k(0) \cdot (Du_\infty(0)-a_k) = 0.
\end{equation*}
Since $\beta_k(0) \to \beta_\infty$ as $k \to \infty$ and $\beta_\infty \cdot Du_\infty(0) = 0$, we deduce that $a_k \to 0$. 
Thus, defining $h_k(x) := u_\infty(x) - a_k \cdot x$, we obtain $\beta_k(0) \cdot Dh_k(0) = 0$ and
\begin{equation*}
    \norm{h_k}_{C^{1, \alpha_{0}}\left(\overline{B_{3/4}^{+}}\right)} \leq \norm{u_{\infty}}_{C^{1, \alpha_{0}}\left(\overline{B_{3/4}^{+}}\right)} + \norm{a_{k}\cdot x}_{C^{1, \alpha_{0}}\left(\overline{B_{3/4}^{+}}\right)} \leq C_e + 2C_e = 3C_e = \ov{C_e}.
\end{equation*}
Finally, observing that
\begin{equation*}
    \norm{u_k-h_k}_{L^\infty\left(\Omega_{1/2}\right)} \leq \norm{u_k-u_\infty}_{L^\infty\left(\Omega_{1/2}\right)} + |a_k| \to 0 \quad \text{as } k \to \infty,
\end{equation*}
we reach a contradiction for sufficiently large $k$, which completes the proof.
\end{proof}

With Lemma \ref{approx} in hand, we are now ready to establish the improvement of flatness. The following proposition shows that if the given data are sufficiently small, the solution can be closely approximated by an affine function in a smaller domain.

\begin{prop}
\label{first step} 
Let $\Omega$ be a bounded $C^{1}$ domain whose boundary is represented as in \eqref{boundary repre}. Assume \ref{a1}--\ref{a3}. Let $i(\Phi)\geq0$ and $0\leq\gamma<1$. Assume that  $f\in C(\Omega)\cap L^{\infty}(\Omega)$, $g\in C^{\alpha}(\partial\Omega)$, and $\beta\in C^{\alpha}(\partial\Omega;\R^{n})$ for some $\alpha\in(0,1)$.
Let $u$ satisfy $|u|\leq1$ in $\Omega_{1}$ and be a viscosity solution to 
\begin{equation*}
\left\{\begin{aligned}
\Phi\left(x, \left|Du-q\right|\right)F(D^{2}u) &=f|Du-q|^{\gamma} &&\text{in } \Omega_{1}, \\
\beta\cdot Du&=g &&\text{on } \partial\Omega_{1}.
\end{aligned}\right.
\end{equation*} 
Given $\ov{\alpha}$ satisfying $$\ov{\alpha}\in(0, \alpha_{0})\cap\left(0, \frac{1}{1+s(\Phi)-\min(i(\Phi),\gamma)}\right]\cap(0, \alpha],$$
there exist constants $0<r<1/2$ and $\epsilon>0$ such that if 
\begin{align*}
(1+|q|^{(\gamma-i(\Phi))_+})\norm{f}_{L^{\infty}}, \; \norm{g}_{L^{\infty}}, \; [\beta]_{C^{\alpha}}, \; [\varphi]_{C^{1}}\leq \epsilon,    
\end{align*}
then there exists an affine function $l=u(0)+b\cdot x$ such that
\begin{align*}
\norm{u-l}_{L^{\infty}(\Omega_{r})}\leq r^{1+\ov{\alpha}}, \; \beta(0)\cdot b=0, \quad \text{and} \quad |b|\leq \ov{C_{e}},    
\end{align*}
where $\ov{C_e}$ is the constant defined in Lemma \ref{approx}.
\end{prop}
\begin{proof}
   For a given $\delta>0$ to be determined later, let $\epsilon>0$ be the constant from Lemma \ref{approx}.
   Then, there exists a function $h\in C^{1, \alpha_{0}}$ such that $\norm{u-h}_{L^{\infty}}\leq \delta$, $\beta(0) \cdot Dh(0)=0$, and $\norm{h}_{C^{1, \alpha_{0}}}\leq \ov{C_e}$.
   Note that $|Dh(0)| \leq \ov{C_e}$. Let $\tilde{l}(x) =h(0)+ Dh(0)\cdot x$.
   It follows that
   \begin{align*}
       |h(x)-\tilde{l}(x)| \leq \ov{C_{e}} |x|^{1+\alpha_0}.
   \end{align*}
   We then select $0<r<\frac{1}{2}$ sufficiently small such that $\ov{C_e} r^{1+\alpha_0} \leq \frac{1}{3}r^{1+\ov{\alpha}}$, and set $\delta = \frac{1}{3}r^{1+\ov{\alpha}}$.
   Defining $l(x) =u(0)+ Dh(0)\cdot x$, we obtain
   \begin{align*}
       \norm{u-l}_{L^{\infty}(\Omega_{r})} &\leq \norm{u-h}_{L^{\infty}(\Omega_{r})} +  \norm{h-\tilde{l}}_{L^{\infty}(\Omega_{r})} +|u(0)-h(0)| \\
       &\leq 2\delta + \ov{C_e}r^{1+\alpha_0} \leq r^{1+\ov{\alpha}},
   \end{align*}
   which completes the proof.
\end{proof}

\section{Proof of Theorem \ref{thm1}}
\label{sec4}

\begin{proof}[Proof of Theorem \ref{thm1}]
In light of Lemma \ref{transform} and Remark \ref{unified}, we may assume that $u$ is a viscosity solution to
\begin{equation*}
\left\{\begin{aligned}
\Phi(x, |Du|)F(D^{2}u) &=f|Du|^{\gamma} &&\text{in } \Omega\cap B_{1}, \\
\beta\cdot Du&=g &&\text{on } \partial\Omega\cap B_{1},
\end{aligned}\right.
\end{equation*}
where $i(\Phi)\geq 0$ and $0\leq \gamma<1$. Let us fix an exponent $\ov{\alpha}$ satisfying
\begin{equation*}
\ov{\alpha}\in (0, \alpha_{0})\cap\left(0, \frac{1}{1+s(\Phi)-\min(i(\Phi), \gamma)}\right]\cap(0, \alpha].
\end{equation*}
Note that, in view of Remark \ref{unified}, this range of $\ov{\alpha}$ precisely corresponds to the range \eqref{alpha range} in Theorem \ref{thm1}.

Let $r$ and $\epsilon$ be the constants given in Proposition \ref{first step}. By a suitable rotation and translation of the coordinates, we may assume that $0 \in \partial\Omega$ and the domain $\Omega_1$ is given by $\{x \in B_1 : x_n > \varphi_\Omega(x')\}$ for some $C^1$ function $\varphi=\varphi_\Omega$ with $\varphi(0) = 0$ and $D\varphi(0) = 0$.

Proceeding as in \cite[Section 5]{BKO25}, we may further assume that $u$ satisfies 
\begin{equation*}
\left\{\begin{aligned}
\Phi(x, |Du-q|)F(D^{2}u) &=f|Du-q|^{\gamma} &&\text{in } \Omega_{1}, \\
\beta\cdot Du&=g &&\text{on } \partial\Omega_{1},
\end{aligned}\right.
\end{equation*}
for some vector $q \in \mathbb{R}^n$, with the constants $\nu_0$ and $\nu_1$ in assumption \ref{a2} normalized to $1$. Additionally, we may assume that the solution satisfies $u(0)=0$ and $\norm{u}_{L^\infty} \leq 1$, the boundary data satisfies $g(0)=0$, and the following smallness conditions hold:
\begin{equation*}
\norm{f}_{L^{\infty}}\leq\frac{1-\left(1/2\right)^{\ov{\alpha}}}{4\ov{C_{e}}M}\epsilon, \quad [g]_{C^{\alpha}}\leq \frac{\epsilon}{2}, \quad [\beta]_{C^{\alpha}}\leq\frac{1-r^{\ov{\alpha}}}{2\ov{C_{e}}}\epsilon, \quad [\varphi]_{C^{1}}\leq\epsilon, \quad \text{and} \quad |q|\leq\frac{\ov{C_{e}}}{1-r^{\ov{\alpha}}}.
\end{equation*}
Indeed, to justify this, let $0<s<1$ and $K>1$ be constants to be determined. We define the rescaled quantities
\begin{align*}
\widetilde{u}(x)&:=\frac{u(sx)-u(0)}{K}-\frac{sg(0)}{K\beta_{n}(0)}x_{n}, \\
\widetilde{\Phi}(x,t)&:=\frac{\Phi(sx, \frac{K}{s}t)}{\Phi(sx, \frac{K}{s})}, \quad \quad q:=-\frac{sg(0)}{K\beta_{n}(0)}e_{n}, \\
\widetilde{F}(M)&:=\frac{s^{2}}{K}F\left(\frac{K}{s^{2}}M\right), \quad \widetilde{f}(x):=\frac{s^{2-\gamma}f(sx)}{K^{1-\gamma}\Phi(sx, \frac{K}{s})}, \\
\widetilde{\beta}(x)&:=\beta(sx), \quad \widetilde{g}(x):=\frac{s}{K}g(sx)-\frac{sg(0)}{K\beta_{n}(0)}\beta_{n}(sx), \\
\widetilde{\varphi}(x)&:=\frac{\varphi(sx)}{s}(=\varphi_{\frac{1}{s}\Omega}(x)).
\end{align*}
Then, by the scaling properties (Section \ref{scaling property}), $\widetilde{u}$ solves
\begin{equation*}
\left\{\begin{aligned}
\widetilde{\Phi}(x, |D\widetilde{u}-q|)\widetilde{F}(D^{2}\widetilde{u}) &=\widetilde{f}|D\widetilde{u}-q|^{\gamma} &&\text{in } \left(\frac{1}{s}\Omega\right)\cap B_{1}, \\
\widetilde{\beta}\cdot D\widetilde{u}&=\widetilde{g} &&\text{on } \partial\left(\frac{1}{s}\Omega\right)\cap B_{1}.
\end{aligned}\right.
\end{equation*}
Since $\varphi\in C^{1}$ and $D\varphi(0)=0$, we can choose $s \in (0,1)$ sufficiently small, depending only on $[\beta]_{C^{\alpha}}$ and the $C^{1}$ modulus of $\varphi$, such that
\begin{equation*}
|D\varphi(x)| \leq \epsilon \quad \text{for any } x\in \ov{B_{s}},   
\end{equation*}
which yields $[\widetilde{\varphi}]_{C^{1}} \leq \epsilon$, and simultaneously,
\begin{equation} \label{estimate of beta holder norm}
[\widetilde{\beta}]_{C^{\alpha}} = s^{\alpha}[\beta]_{C^{\alpha}} \leq \frac{1-r^{\ov{\alpha}}}{2\ov{C_{e}}}\epsilon.
\end{equation}
Next, we define $K > 1$ by
\begin{equation*}
K := 1+4\norm{u}_{L^{\infty}}+\left(\frac{4\ov{C_{e}}M^{2}}{(1-(1/2)^{\ov{\alpha}})\nu_{0}\epsilon}\norm{f}_{L^{\infty}}\right)^{1/ (1-\gamma+i(\Phi))}+\left(\frac{2}{\epsilon}+\frac{2}{\delta_{0}}\right)\norm{g}_{C^{\alpha}}.
\end{equation*}
Noting that 
\begin{equation}
\label{beta(0)}
\beta(0)\cdot \mathbf{n}(0)=\beta_{n}(0)\geq\delta_{0},
\end{equation}
we clearly have $\widetilde{u}(0)=0$ and
\begin{equation*}
\norm{\widetilde{u}}_{L^{\infty}}\leq\frac{2\norm{u}_{L^{\infty}}}{K}+ \frac{\norm{g}_{L^{\infty}}}{K\delta_{0}}\leq 1.
\end{equation*}
Furthermore, since $K/s \geq 1$, \ref{a2} implies
\begin{equation*}
\Phi\left(sx, \frac{K}{s}\right)\geq \frac{1}{M}\left(\frac{K}{s}\right)^{i(\Phi)}\Phi(sx,1)\geq \frac{\nu_{0}}{M}\left(\frac{K}{s}\right)^{i(\Phi)}.
\end{equation*}
Using this bound, we deduce
\begin{equation*}
\norm{\widetilde{f}}_{L^{\infty}}\leq \frac{Ms^{2-\gamma+i(\Phi)}}{\nu_{0}K^{1-\gamma+i(\Phi)}}\norm{f}_{L^{\infty}} \leq \frac{M}{\nu_{0}K^{1-\gamma+i(\Phi)}}\norm{f}_{L^{\infty}}\leq \frac{1-(1/2)^{\ov{\alpha}}}{4\ov{C_{e}}M}\epsilon.
\end{equation*}
For the boundary data, we have $\widetilde{g}(0)=0$, and combining \eqref{beta(0)} and \eqref{estimate of beta holder norm}, we obtain
\begin{align*}
[\widetilde{g}]_{C^{\alpha}} &\leq \frac{s^{\alpha+1}}{K}[g]_{C^{\alpha}}+\frac{s^{\alpha+1}\norm{g}_{L^{\infty}}}{K\delta_{0}}[\beta]_{C^{\alpha}} \\
&\leq \frac{1}{K}\left(1+\frac{\epsilon}{2\ov{C_{e}}\delta_{0}}\right)\norm{g}_{C^{\alpha}} \leq \frac{\epsilon}{2}.
\end{align*} 
Finally, by \eqref{beta(0)}, we have
\begin{equation*}
|q|\leq \frac{\norm{g}_{L^{\infty}}}{K\delta_{0}}\leq \frac{1}{2} \leq \frac{\ov{C_{e}}}{1-r^{\ov{\alpha}}}.
\end{equation*}

We now claim that there exists a sequence of linear functions $l_{k}(x)=b_{k}\cdot x$ with $b_{k} \in \mathbb{R}^{n}$, such that for every integer $k \geq 0$:
\begin{equation} \label{induction}
\begin{cases}
\norm{u-l_{k}}_{L^{\infty}\left(\Omega_{r^{k}}\right)}\leq r^{k(1+\ov{\alpha})}, \\
\beta(0)\cdot b_{k} = 0, \\
|b_{k}-b_{k+1}|\leq \ov{C_{e}}r^{k\ov{\alpha}}.
\end{cases}
\end{equation}
We proceed by induction. For $k=0$, the claim holds trivially by choosing $l_{0}=0$, since $\norm{u}_{L^{\infty}}\leq 1$.
Suppose that the induction hypotheses \eqref{induction} hold up to some integer $k \geq 0$. We consider the normalized function
\begin{equation*}
v_k(x) := \frac{(u-l_{k})(r^{k}x)}{r^{k(1+\ov{\alpha})}}.   
\end{equation*}
By the scaling introduced in Section \ref{scaling property}, $v_{k}$ solves
\begin{equation*}
\left\{\begin{aligned}
    \Phi_{k}\left(x,\left|Dv_{k}-q_k\right|\right)F_{k}(D^{2}v_{k})&=f_{k}\left|Dv_{k}-q_k \right|^{\gamma} &&\text{in } \left(\frac{1}{r^{k}}\Omega\right)\cap B_{1},\\
    \beta_{k}\cdot Dv_{k}&=g_{k} &&\text{on } \partial\left(\frac{1}{r^{k}}\Omega\right)\cap B_{1},
\end{aligned}\right.
\end{equation*}
where the rescaled quantities are given by
\begin{align*}
\Phi_{k}(x,t) &:= \frac{\Phi(r^{k}x, r^{k\ov{\alpha}}t)}{\Phi(r^{k}x, r^{k\ov{\alpha}})}, \quad \quad q_k := \frac{q-b_{k}}{r^{k\ov{\alpha}}}, \\
F_{k}(M) &:= r^{k(1-\ov{\alpha})}F(r^{k(\ov{\alpha}-1)}M), \quad f_{k}(x) := \frac{r^{k(1-\ov{\alpha}(1-\gamma))}f(r^{k}x)}{\Phi(r^{k}x, r^{k\ov{\alpha}})}, \\
\beta_{k}(x) &:= \beta(r^{k}x), \quad \quad \quad \quad \quad \quad g_{k}(x) := \frac{g(r^{k}x)-\beta(r^{k}x)\cdot b_{k}}{r^{k\ov{\alpha}}}, \\
\varphi_{k}(x) &:= \frac{\varphi(r^{k}x)}{r^{k}}\left(=\varphi_{\frac{1}{r^{k}}\Omega}\right).
\end{align*}
 In order to apply Proposition \ref{first step}, we now verify its hypotheses for $v_{k}$. We have $\norm{v_{k}}_{L^\infty} \leq 1$ by the induction hypothesis. Since $r^{k\ov{\alpha}}\leq 1$, \ref{a2} yields
\begin{equation*}
\Phi(r^{k}x, r^{k\ov{\alpha}})\geq\frac{1}{M}(r^{k\ov{\alpha}})^{s(\Phi)}\Phi(r^{k}x, 1)=\frac{1}{M}r^{k\ov{\alpha}s(\Phi)}.   
\end{equation*}
Using this bound, we can estimate the source term as
\begin{equation} \label{f_k bound}
    \norm{f_{k}}_{L^{\infty}} \leq M r^{k(1-\ov{\alpha}(1+s(\Phi)-\gamma))} \norm{f}_{L^\infty}.
\end{equation}
Moreover, the induction hypothesis $|b_{j}-b_{j+1}| \leq \ov{C_{e}}r^{j\ov{\alpha}}$ implies
\begin{equation} \label{estimate of b_k}
|b_{k}|\leq\frac{\ov{C_{e}}}{1-r^{\ov{\alpha}}}.   
\end{equation}
Combining this with $|q| \leq \frac{\ov{C_{e}}}{1-r^{\ov{\alpha}}}$, we can bound $q_k$ by
\begin{equation*}
|q_k| = \frac{|q-b_{k}|}{r^{k\ov{\alpha}}} \leq \frac{|q| + |b_k|}{r^{k\ov{\alpha}}} \leq \frac{2\ov{C_{e}}}{r^{k\ov{\alpha}}(1-r^{\ov{\alpha}})}.
\end{equation*}
Consequently, combining this bound with \eqref{f_k bound} gives
\begin{align*}
(1+|q_k|^{(\gamma-i(\Phi))_+})\norm{f_k}_{L^\infty} &\leq  2\left(\frac{2\ov{C_{e}}}{r^{k\ov{\alpha}}(1-r^{\ov{\alpha}})}\right)^{(\gamma-i(\Phi))_{+}} M r^{k(1-\ov{\alpha}(1+s(\Phi)-\gamma))} \norm{f}_{L^\infty}\\
&\leq \frac{4\ov{C_{e}}M}{1-(1/2)^{\ov{\alpha}}} r^{k(1-\ov{\alpha}(1+s(\Phi)-\gamma+(\gamma-i(\Phi))_+))}  \norm{f}_{L^\infty} \\
&= \frac{4\ov{C_{e}}M}{1-(1/2)^{\ov{\alpha}}}r^{k(1-\ov{\alpha}(1+s(\Phi)-\min(i(\Phi), \gamma)))}  \norm{f}_{L^\infty}  \\
&\leq \frac{4\ov{C_{e}}M}{1-(1/2)^{\ov{\alpha}}}\norm{f}_{L^\infty} \leq \epsilon,
\end{align*}
where we used the identity $1+s(\Phi)-\gamma+(\gamma-i(\Phi))_{+} = 1+s(\Phi)-\min(i(\Phi), \gamma)$, and the second-to-last inequality follows from our choice of $\ov{\alpha}$. Regarding the boundary data, $g(0)=0$ and $\beta(0)\cdot b_{k}=0$ together with \eqref{estimate of b_k} yield
\begin{align*}
|g_{k}(x)| &= \frac{|g(r^{k}x)-g(0)-(\beta(r^{k}x)-\beta(0))\cdot b_{k}|}{r^{k\ov{\alpha}}} \\ 
&\leq  \frac{[g]_{C^{\alpha}}|r^{k}x|^{\alpha}+[\beta]_{C^{\alpha}}|r^{k}x|^{\alpha}|b_{k}|}{r^{k\ov{\alpha}}} \\
&\leq [g]_{C^{\alpha}}+[\beta]_{C^{\alpha}}\frac{\ov{C_{e}}}{1-r^{\ov{\alpha}}} \leq \epsilon.
\end{align*}
It is also immediate that $[\beta_{k}]_{C^{\alpha}}\leq r^{k\alpha}[\beta]_{C^{\alpha}}\leq\epsilon$ and $[\varphi_{k}]_{C^{1}} \leq [\varphi]_{C^{1}}\leq \epsilon$.

Thus, all hypotheses of Proposition \ref{first step} are satisfied for $v_k$. Hence, there exists a linear function $\Tilde{l}(x)=\Tilde{b}\cdot x$ such that
\begin{equation*}
\norm{v_{k}-\Tilde{l}}_{L^{\infty}\left(\left(\frac{1}{r^{k}}\Omega\right)_{r}\right)}\leq r^{1+\ov{\alpha}}, \quad \beta_{k}(0)\cdot \Tilde{b} = \beta(0)\cdot \Tilde{b} = 0, \quad \text{and} \quad |\Tilde{b}|\leq \ov{C_{e}}.    
\end{equation*}
Setting 
\begin{equation*}
l_{k+1}(x) := l_{k}(x)+r^{k(1+\ov{\alpha})}\Tilde{l}(x/r^{k}),
\end{equation*}
we verify that \eqref{induction} holds for $k+1$. This completes the induction step, and the theorem follows.
\end{proof}

\vspace{0.2cm}

{\bf Acknowledgments.} S. Byun was supported by Mid-Career Bridging Program through Seoul National University. H. Kim and S. Kim were supported the National Research Foundation of Korea grant funded by the Korean Government
(NRF-2022R1A2C1009312).

\section*{Data Availability}
Data sharing not applicable to this article as no datasets were generated or analysed during the current study.

\section*{Declarations}

\subsection*{Conflicts of Interest}
The authors have no conflicts of interest to declare that are relevant to the content of this article.

\bibliographystyle{amsplain}
\bibliography{ref}

@article {IS2013,
    AUTHOR = {Imbert, C. and Silvestre, L.},
     TITLE = {{$C^{1,\alpha}$} regularity of solutions of some degenerate
              fully non-linear elliptic equations},
   JOURNAL = {Adv. Math.},
  FJOURNAL = {Advances in Mathematics},
    VOLUME = {233},
      YEAR = {2013},
     PAGES = {196--206},
      ISSN = {0001-8708,1090-2082},
   MRCLASS = {35J60 (35B65 35D40)},
  MRNUMBER = {2995669},
MRREVIEWER = {Barbara\ Brandolini},
       DOI = {10.1016/j.aim.2012.07.033},
       URL = {https://doi.org/10.1016/j.aim.2012.07.033},
}

@article {BV2022,
    AUTHOR = {Banerjee, Agnid and Verma, Ram Baran},
     TITLE = {{$C^{1,\alpha}$} regularity for degenerate fully nonlinear
              elliptic equations with {N}eumann boundary conditions},
   JOURNAL = {Potential Anal.},
  FJOURNAL = {Potential Analysis. An International Journal Devoted to the
              Interactions between Potential Theory, Probability Theory,
              Geometry and Functional Analysis},
    VOLUME = {57},
      YEAR = {2022},
    NUMBER = {3},
     PAGES = {327--365},
      ISSN = {0926-2601,1572-929X},
   MRCLASS = {35J60 (35D40)},
  MRNUMBER = {4482106},
MRREVIEWER = {Chao\ Zhang},
       DOI = {10.1007/s11118-021-09918-z},
       URL = {https://doi.org/10.1007/s11118-021-09918-z},
}

@article {P2008,
    AUTHOR = {Patrizi, Stefania},
     TITLE = {The {N}eumann problem for singular fully nonlinear operators},
   JOURNAL = {J. Math. Pures Appl. (9)},
  FJOURNAL = {Journal de Math\'{e}matiques Pures et Appliqu\'{e}es.
              Neuvi\`eme S\'{e}rie},
    VOLUME = {90},
      YEAR = {2008},
    NUMBER = {3},
     PAGES = {286--311},
      ISSN = {0021-7824},
   MRCLASS = {35J60 (35B25 35B50)},
  MRNUMBER = {2446081},
MRREVIEWER = {Kai\ Seng\ Chou},
       DOI = {10.1016/j.matpur.2008.04.007},
       URL = {https://doi.org/10.1016/j.matpur.2008.04.007},
}

@article {BBLL24a,
    AUTHOR = {Baasandorj, Sumiya and Byun, Sun-Sig and Lee, Ki-Ahm and Lee,
              Se-Chan},
     TITLE = {{$C^{1,\alpha}$}-regularity for a class of degenerate/singular
              fully non-linear elliptic equations},
   JOURNAL = {Interfaces Free Bound.},
  FJOURNAL = {Interfaces and Free Boundaries. Mathematical Analysis,
              Computation and Applications},
    VOLUME = {26},
      YEAR = {2024},
    NUMBER = {2},
     PAGES = {189--215},
      ISSN = {1463-9963,1463-9971},
   MRCLASS = {35B65 (35D40 35J60 35J70)},
  MRNUMBER = {4733905},
MRREVIEWER = {Zhijiun\ Zhang},
       DOI = {10.4171/ifb/496},
       URL = {https://doi.org/10.4171/ifb/496},
}

@article {IS2016,
    AUTHOR = {Imbert, Cyril and Silvestre, Luis},
     TITLE = {Estimates on elliptic equations that hold only where the
              gradient is large},
   JOURNAL = {J. Eur. Math. Soc. (JEMS)},
  FJOURNAL = {Journal of the European Mathematical Society (JEMS)},
    VOLUME = {18},
      YEAR = {2016},
    NUMBER = {6},
     PAGES = {1321--1338},
      ISSN = {1435-9855,1435-9863},
   MRCLASS = {35J60 (35B45 35B65 35D40 35J70)},
  MRNUMBER = {3500837},
MRREVIEWER = {Barbara\ Brandolini},
       DOI = {10.4171/JEMS/614},
       URL = {https://doi.org/10.4171/JEMS/614},
}

@article {BBLL24b,
    AUTHOR = {Baasandorj, Sumiya and Byun, Sun-Sig and Lee, Ki-Ahm and Lee,
              Se-Chan},
     TITLE = {Global regularity results for a class of singular/degenerate
              fully nonlinear elliptic equations},
   JOURNAL = {Math. Z.},
  FJOURNAL = {Mathematische Zeitschrift},
    VOLUME = {306},
      YEAR = {2024},
    NUMBER = {1},
     PAGES = {Paper No. 1, 26},
      ISSN = {0025-5874,1432-1823},
   MRCLASS = {35B65 (35D40 35J60 35J70)},
  MRNUMBER = {4670092},
MRREVIEWER = {Dami\~ao\ Ara\'ujo},
       DOI = {10.1007/s00209-023-03400-9},
       URL = {https://doi.org/10.1007/s00209-023-03400-9},
}

@article {LZ2018,
    AUTHOR = {Li, Dongsheng and Zhang, Kai},
     TITLE = {Regularity for fully nonlinear elliptic equations with oblique
              boundary conditions},
   JOURNAL = {Arch. Ration. Mech. Anal.},
  FJOURNAL = {Archive for Rational Mechanics and Analysis},
    VOLUME = {228},
      YEAR = {2018},
    NUMBER = {3},
     PAGES = {923--967},
      ISSN = {0003-9527,1432-0673},
   MRCLASS = {35J60 (35B65 35D40 35J25 49L25)},
  MRNUMBER = {3780142},
MRREVIEWER = {Tilak\ Bhattacharya},
       DOI = {10.1007/s00205-017-1209-x},
       URL = {https://doi.org/10.1007/s00205-017-1209-x},
}

@article {MS2006,
    AUTHOR = {Milakis, Emmanouil and Silvestre, Luis E.},
     TITLE = {Regularity for fully nonlinear elliptic equations with
              {N}eumann boundary data},
   JOURNAL = {Comm. Partial Differential Equations},
  FJOURNAL = {Communications in Partial Differential Equations},
    VOLUME = {31},
      YEAR = {2006},
    NUMBER = {7-9},
     PAGES = {1227--1252},
      ISSN = {0360-5302,1532-4133},
   MRCLASS = {35J60 (35B65)},
  MRNUMBER = {2254613},
MRREVIEWER = {Eduardo\ V.\ Teixeira},
       DOI = {10.1080/03605300600634999},
       URL = {https://doi.org/10.1080/03605300600634999},
}

@article {BKO25,
    AUTHOR = {Byun, Sun-Sig and Kim, Hongsoo and Oh, Jehan},
     TITLE = {{$C^{1,\alpha }$} regularity for degenerate fully nonlinear
              elliptic equations with oblique boundary conditions on {$C^1$}
              domains},
   JOURNAL = {Calc. Var. Partial Differential Equations},
  FJOURNAL = {Calculus of Variations and Partial Differential Equations},
    VOLUME = {64},
      YEAR = {2025},
    NUMBER = {5},
     PAGES = {Paper No. 174, 20},
      ISSN = {0944-2669,1432-0835},
   MRCLASS = {35J60 (35B65 35D40 35J25 35J70)},
  MRNUMBER = {4913059},
MRREVIEWER = {Qi\ Li},
       DOI = {10.1007/s00526-025-03042-1},
       URL = {https://doi.org/10.1007/s00526-025-03042-1},
}

@article {BRS26,
    AUTHOR = {Bessa, Junior da S. and Ricarte, Gleydson C. and Silva, Paulo
              H. da C.},
     TITLE = {Optimal gradient regularity to degenerate fully nonlinear
              elliptic models with oblique boundary condition},
   JOURNAL = {Nonlinear Anal.},
  FJOURNAL = {Nonlinear Analysis. Theory, Methods \& Applications. An
              International Multidisciplinary Journal},
    VOLUME = {262},
      YEAR = {2026},
     PAGES = {Paper No. 113919, 16},
      ISSN = {0362-546X,1873-5215},
   MRCLASS = {35J60 (35B65 35J15 35J25 35R35)},
  MRNUMBER = {4948058},
       DOI = {10.1016/j.na.2025.113919},
       URL = {https://doi.org/10.1016/j.na.2025.113919},
}

@article {R20,
    AUTHOR = {Ricarte, Gleydson C.},
     TITLE = {Optimal {$C^{1,\alpha}$} regularity for degenerate fully
              nonlinear elliptic equations with {N}eumann boundary
              condition},
   JOURNAL = {Nonlinear Anal.},
  FJOURNAL = {Nonlinear Analysis. Theory, Methods \& Applications. An
              International Multidisciplinary Journal},
    VOLUME = {198},
      YEAR = {2020},
     PAGES = {111867, 13},
      ISSN = {0362-546X,1873-5215},
   MRCLASS = {35J60 (35B65 35J75 35R35)},
  MRNUMBER = {4081861},
       DOI = {10.1016/j.na.2020.111867},
       URL = {https://doi.org/10.1016/j.na.2020.111867},
}

@article {BDF22,
    AUTHOR = {Birindelli, Isabeau and Demengel, Fran\c coise and Leoni,
              Fabiana},
     TITLE = {Mixed boundary value problems for fully nonlinear degenerate
              or singular equations},
   JOURNAL = {Nonlinear Anal.},
  FJOURNAL = {Nonlinear Analysis. Theory, Methods \& Applications. An
              International Multidisciplinary Journal},
    VOLUME = {223},
      YEAR = {2022},
     PAGES = {Paper No. 113006, 22},
      ISSN = {0362-546X,1873-5215},
   MRCLASS = {35J66 (35J70 35J75)},
  MRNUMBER = {4438232},
       DOI = {10.1016/j.na.2022.113006},
       URL = {https://doi.org/10.1016/j.na.2022.113006},
}

@article{BO26,
      title={Optimal {$C^{1,\alpha}$}  regularity up to the boundary for fully nonlinear elliptic equations with double phase degeneracy}, 
      author={Junior da Silva Bessa and Jehan Oh},
      year={2026},
      JOURNAL={arXiv preprint},
      pages = {2604.04776},
      primaryClass={math.AP},
      url={https://arxiv.org/abs/2604.04776}, 
}

\end{document}